\documentclass[12 pt]{amsart}
\pdfoutput=1
\usepackage{amsmath}
\usepackage{amsfonts}
\usepackage[pdftex]{graphicx}
\DeclareGraphicsExtensions{.png,}
\newtheorem{lemma}{Lemma}
\newtheorem{theorem}{Theorem}
\newtheorem{prop}{Proposition}

\newcommand{\R}{\mathbb{R}}

\newcommand{\N}{\mathbb{N}}

\renewcommand{\dim}{\mbox{dim}_{\mathcal{H}}}
\title[Self-affine multifractal analysis]{Multifractal analysis for Bedford-McMullen carpets}
\author{Thomas Jordan}
\address{Thomas Jordan\\Department of Mathematics\\ The University of Bristol\\
University Walk\\Clifton\\ Bristol\\BS8 1TW\\UK}
\email{thomas.jordan@bristol.ac.uk}
\author{Micha\l\ Rams}
\address{Micha\l\ Rams\\Institute of Mathematics\\ Polish Academy of Sciences\\ ul.
\'Sniadeckich 8, 00-956 Warszawa\\ Poland }
\email{M.Rams@impan.gov.pl}
\thanks{The research of M.R. was supported by grants EU FP6 ToK SPADE2, EU FP6 RTN
CODY and MNiSW grant 'Chaos, fraktale i dynamika konforemna'. The
research was started during a visit of M. R. to Bristol. M. R.
would like to thank Bristol University for the hospitality shown
during his visit.}
\begin{document}
\begin{abstract}
In this paper we compute the multifractal analysis for local
dimensions of Bernoulli measures supported on the self-affine
carpets introduced by Bedford-McMullen. This extends the work of
King where the multifractal analysis is computed with strong additional separation
assumptions.
\end{abstract}
\maketitle
\section{Introduction}

The multifractal properties of local dimensions of fractal
measures have been studied for more than twenty years. Some problems are
already completely solved (for example, the local dimension spectra
for Gibbs measures on a conformal repeller, see \cite{pesinbook}).
However non-conformal systems turn out to be much less
tractable and only some specific examples have been solved.

By the local dimension spectrum of a measure we mean the function $\alpha\rightarrow\dim X_{\alpha}$ where $X_{\alpha}$ is the set of all points with local dimension $\alpha$.
There exists a well-developed technique for calculating the local
dimension spectra of invariant measures for dynamical systems.
One begins by introducing a symbolic description on the
attractor and defining a suitable {\it symbolic} local dimension of
the measure. The first step is usually first to compute the multifractal spectrum for the symbolic local dimension. This is done by constructing a suitable auxiliary measure which is exact-dimensional and only supported on the set of points where
the symbolic local dimension of the original measure takes a prescribed
value, say $\alpha$. If the auxiliary measure has been correctly chosen then the Hausdorff dimension of points with symbolic local dimension $\alpha$ will be the dimension of the auxiliary measure. The final step is to show that the multifractal spectrum for the symbolic local dimension of the original measure is the same as the spectrum for the real local dimension. Once
again, we refer the reader to \cite{pesinbook},
where this technique is explained in detail.

In this paper we are interested in the local dimension spectra
for Bernoulli measures on Bedford-McMullen carpets. This
problem has already been studied in several papers. King
calculated the symbolic local dimension spectrum in \cite{king},
this result was then generalized to Gibbs measures by Barral and
Mensi \cite{BM} and to higher dimensional generalized Sierpi\'nski
carpets by Olsen \cite{olsen} (see also \cite{olsen2}). However,
in all those papers authors were unable to make the last step,
from symbolic to real local dimension. For that reason, these
papers had to assume some strong additional assumptions
(e.g. a very strong separation property) guaranteeing the equality
of symbolic and real local dimension spectra. However in \cite{BM} it is shown the symbolic and real local dimension spectra are the same for the decreasing part of the spectrum. Naturally, it was
conjectured (for example by Olsen in \cite{olsen2}) that these
additional assumptions are not necessary for the increasing part of the spectra also to be the same.

The purpose of this paper is to present this missing argument.
Unfortunately our arguments only apply to the two-dimensional case.
For simplicity we just look at the Bernoulli measures studied in \cite{king}.

\section{Statement of results}

We now proceed to formally state our results. To define the
Bedford-McMullen carpets \cite{bed, mcmullen} we introduce a digit
set
$$D\subseteq\{0,\ldots m-1\}\times\{0,\ldots,n-1\},$$
where $m<n$. For each $(i,j)\in D$ we define
$T_{i,j}:\R^2\rightarrow\R^2$ by
$T_{i,j}(x,y)=(n^{-1}x,m^{-1}y)+(j,i)$. We let $\Lambda$ be the
unique non-empty compact set which satisfies $\cup_{(i,j)\in
D}T_{i,j}(\Lambda)=\Lambda$. We will also let $\sigma=\frac{\log
m}{\log n}$. The Hausdorff and box counting dimension of $\Lambda$
were calculated by McMullen and Bedford in \cite{mcmullen} and
\cite{bed}. We introduce a positive probability vector
$\underline{p}$ with element $p_{ij}$ for each $(i,j)\in D$. We
also define the related  probability vector $\underline{q}$ where
$q_i=\sum_{j:(i,j)\in D} p_{ij}$. Thus we can define a self-affine
measure $\mu$ which is the unique probability measure satisfying
$$\mu(A)=\sum_{(i,j)\in D} p_{ij}\mu(T_{i,j}^{-1} A)$$
for any Borel subset of $\R^2$. For any $x\in\R^2$ the local
dimension of $\mu$ at $x$ is defined by
$$d_{\mu}(x)=\lim_{r\rightarrow 0}\frac{\log\mu(B(x,r))}{\log r}$$
if this limit exists. For $\alpha\in\R$ the level sets $X_{\alpha}$ are defined by
$$X_{\alpha}=\{x\in\Lambda:d_{\mu}(x)=\alpha\}.$$
In King \cite{king} the function $f(\alpha)=\dim X_\alpha$ was
calculated under the condition that for any pair
$\{(i,j),(i',j')\}\in D$ we have that $|i-i'|\neq 1$ and if $i=i'$
then $|j-j'|>1$. In \cite{BM} this assumption is weakened: they
are also able to allow digit sets $D$ where
$D\cap\{(0,0),\ldots,(0,n-1)\}=\emptyset$ or
$D\cap\{(m-1,0),\ldots,(m-1,n-1)\}=\emptyset$. Furthermore
arbitrary digit sets can be considered as long as the probability
vector $p$ is chosen so that $\sum_{j:(1,j)\in D}
p_{1j}^t=\sum_{j:(m,j)\in D}p_{mj}^t$ for all $t>0$. We will only
assume that the digit set has elements in more than one row and
more than one column, without this assumption the measure $\mu$ is
effectively a self-similar measure on the line and the singularity
spectrum is computed in \cite{AP}.

The formula we obtain for $\dim X_{\alpha}$ is exactly the formula
obtained in \cite{king}. To recall the definition we fix $t>0$ and
let $\gamma_i=\sum_{j:(i,j)\in D} p_{ij}^t$. We then define
$\beta(t)$ to be the unique solution to
$$m^{\beta(t)}\sum_{(i,j)\in D}p_{ij}^tq_i^{(1-\sigma)t}\gamma_i^{\sigma-1}=1.$$
We will let
$$\alpha_{\min}=\min_{(i,j)\in D}\frac{ -\sigma\log p_{ij}+(\sigma-1)\log q_i}{\log m}$$
and
$$\alpha_{\max}=\max_{(i,j)\in D}\frac{ -\sigma\log p_{ij}+(\sigma-1)\log q_i}{\log m}.$$
Our main result is that
\begin{theorem}\label{main}
For any $\alpha\in (\alpha_{\min},\alpha_{\max})$ we have that
$$f(\alpha)=\dim X_{\alpha}=\inf_{t}(\alpha t+\beta(t)).$$
In other words $f$ is the Legendre transform of $\beta$.
Furthermore $f$ is differentiable with respect to $\alpha$ and is
concave.
\end{theorem}
\begin{figure}
\includegraphics[width=3in,height=3in]{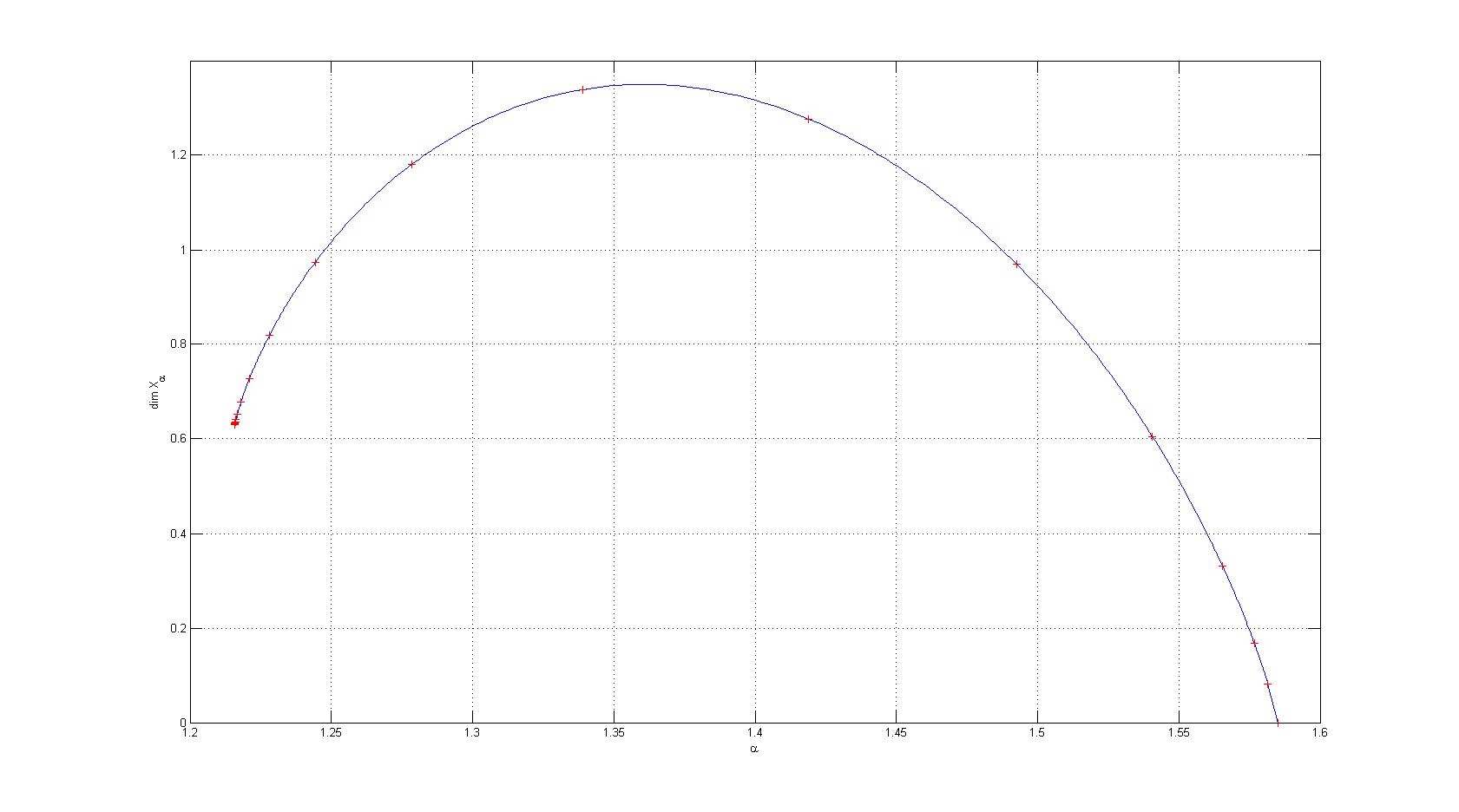}
\caption{The graph of $\alpha\rightarrow\dim X_{\alpha}$ where
$m=2,n=3$, $D=\{(0,0),(0,2),(1,1)\}$ and
$\underline{p}=\left(\frac{1}{3},\frac{1}{3},\frac{1}{3}\right)$. The endpoints of the graph are $\left(\frac{\log 3}{\log 2}+\frac{\log 2}{\log 3}-1,\frac{\log 2}{\log 3}\right)$ and $\left(\frac{\log 3}{\log 2},0\right)$.}
\end{figure}
We are not going to rewrite all of the King's paper \cite{king},
so we will frequently make use of his partial results, referring
the reader to his paper. In particular, the properties of
$f(\alpha)$ can be found in section 4 of \cite{king}.

The rest of paper is divided as follows. In Section \ref{lb} we
will show how to obtain the lower bound for arbitrarily digit
sets. The main argument, calculating of the upper bound, is
presented in Section \ref{ub}.

\section{Symbolic coding and the lower bound}\label{lb}

The lower bound we need was obtained by Barral and Mensi in
\cite{BM} however for completeness we give a proof here. The lower
bound can be proved following the method of King, \cite{king},
with the addition of just one simple lemma. First of all we need
to introduce of a natural symbolic coding. If we let
$\Sigma=D^{\N}$ it is possible to define a natural projection
$\Pi:\Sigma\rightarrow\Lambda$. Initially we let
$\pi:D\rightarrow\{0,\ldots,m-1\}$ be defined by $\pi((i,j))=j$
and $\overline{\pi}:D\rightarrow\{0,\ldots,n-1\}$ by
$\overline{\pi}(i,j)=i$. This allows us to define $\Pi$ by
$$\Pi(\underline{i})=\sum_{j=0}^{\infty}(\overline{\pi}(i_j)n^{-j},\pi(i_j)m^{-j}),$$
where $i_j$ is the $j$-th element of the sequence $\underline{i}$.

It is usual in the study of self-similar sets or conformal systems
to study cylinder sets. In this self-affine setting there is a
related idea of approximate squares. To construct them we let
$\overline{\Sigma}=\{0,\ldots,m-1\}^\N$ and
$\pi,\overline{\pi}:D\rightarrow\{0,\ldots,m-1\}$ be defined by
$\pi((i,j))=j$ and $\overline{\pi}(i,j)=i$. For $k\in\N$ we let
$l(k)=[\sigma k]$ (where $[.]$ denotes the integer part). For
$\underline{i}\in \Sigma$ we define
$$R_k(\underline{i})=\left[\sum_{j=1}^{l(k)}\overline{\pi}(i_j)n^{-j}, \sum_{j=1}^{l(k)}\overline{\pi}(i_j)n^{-j}+n^{-l(k)}\right] \times  \left[\sum_{j=1}^{k}\pi(i_j)m^{-j},
\sum_{j=1}^{k}\pi(i_j)m^{-j}+m^{-k}\right]      .$$ The shorter
side of $R_k(\underline{i})$ is always of length $m^{-k}$. As the
ratio between the sides of this rectangle is clearly between $1:1$
and $1:n$, we can let $D_1=\sqrt{n^2+1}$ and note that for all
$\underline{i}\in\Sigma$ we have that $|R_k(\underline{i})|\leq
D_1m^{-k}$. An approximate square $R_k(\underline{i})$ can contain
elements $x=\Pi(\underline{j})$ where $(j_{l+1},\ldots,j_k)\neq
(i_{l+1},\ldots,i_k)$. To help keep track of these elements we
denote the set of their possible initial segments of symbolic
expansions:
$$\Gamma_k(\underline{i})=\{(j_1,\ldots,j_k) :(i_1,\ldots,i_l)=(j_1,\ldots,j_l)\text{ and }\pi(j_{l+1})=\pi(i_{l+1}),\ldots,\pi(j_k)=\pi(i_k)\}.$$

For each $\underline{i}\in\Sigma$ we will define
$$\delta_{\mu}(\underline{i})=\lim_{k\rightarrow\infty}-\frac{\log\mu(R_k(\underline{i}))}{k\log m}$$
if this limit exists. Under the separation conditions imposed in \cite{king} if $\Pi\underline{i}=x$ then
$\delta_{\mu}(\underline{i})=d_{\mu}(x)$ as long as one of these two limits exists. However without these
assumptions this is not always the case. The following lemma shows that in terms of calculating the
lower bound this poses no problem. We fix $t\in R$ and let $\nu_t$ be the Bernoulli measure defined by the
probability vector $\{P_{ij}\}_{(i,j)\in D}$ where for $(i,j)\in D$
$$P_{ij}=p_{ij}^tm^{\beta(t)}q_i^{(1-\sigma)t}\gamma_i^{\sigma-1}.$$
We also let $Q_i=\sum_{j:(i,j)\in D}P_{ij}$ and
$$\alpha(t)=\frac{-\sigma\sum_{(i,j)\in D}P_{ij}\log p_{ij}-(1-\sigma)\sum_{i=0}^{m-1} Q_i\log q_i}{\log m}.$$

\begin{lemma}
For $\mu_t$ almost all $\underline{i}$ we have that
$$\delta_{\mu}(\underline{i})=d_{\mu}(\Pi\underline{i})=\alpha(t)$$
\end{lemma}
\begin{proof}
The fact that $\delta_{\mu}(\underline{i})=\alpha(t)$ for $\mu_t$
almost all $\underline{i}$ follows from Lemma 4 of \cite{king} (in
the formula given for $\alpha(t)$ in lemma 5 in \cite{king}
$-\sum_{(i,j)\in D} p_{ij}$ should read $-\sum_{(i,j)\in D}
P_{ij}\log P_{ij}$). Thus we only need to show that
$\delta_{\mu}(\underline{i})=d_{\mu}(\Pi\underline{i})$ for
$\mu_t$ almost all $\underline{i}$. We let
$$Z_k(\underline{i})=\max\left\{\min\{\eta:\pi(i_{k+\eta})\neq \pi(i_{k+1})\},
\sigma^{-1}\min\{\eta:\overline{\pi}(i_{l(k)+\eta})\neq \overline{\pi}(i_{l(k)})\}\right\}$$
and note that by the Borel-Cantelli Lemma
\[
\limsup_{k\rightarrow\infty}\frac{Z_k(\underline{i})}{k}=0
\]
for $\mu_t$ almost all $\underline{i}$ (here we use the assumption
that the digit set has elements in two distinct rows and two
distinct columns). We now fix $\underline{i}\in\Sigma$ such that
$\limsup_{k\rightarrow\infty}\frac{Z_k(\underline{i})}{k}=0$ and
let $x=\Pi(\underline{i})$. We have that for any $k$
$$\mu(B(x,D_1m^{-k}))\geq\mu(R_k(\underline{i}))\geq \mu(B(x,m^{-k-Z_k(\underline{i})})).$$
The result now follows by considering $k$ sufficiently large.
\end{proof}

We can now conclude that
$$\dim(X_{\alpha(t)})\geq\dim \mu_t\circ\Pi^{-1}.$$
However in \cite{king} it is shown in Lemma 5 that
$\alpha(t)=-\beta'(t)$ and that
$\dim\mu_t\circ\Pi^{-1}=t\alpha(t)+\beta(t)$. Technically this
result in King is proved with additional assumptions about $D$
however the formula for the dimension of a Bernoulli measure is
still valid without these assumption (see for example \cite{bed}
or \cite{BM}). We can now deduce that
$$\dim(X_{\alpha(t)})\geq t\alpha(t)+\beta(t)$$
and the proof of the lower bound is complete.

\section{Upper bound}\label{ub}

In this section we will find efficient coverings of $X_{\alpha}$
by considering appropriate sets of approximate squares. In
particular we want to show that $\dim X_{\alpha}\leq \alpha
t+\beta(t)$ for any $t\in\R$. Combining this with the lower bound
completes the proof of Theorem \ref{main}. We will fix $t\in\R$
for the rest of this section. For any $\alpha\in
(\alpha_{\min},\alpha_{\max})$ and $\epsilon>0$ we will denote
$$Y(\alpha,\epsilon,k)=\{R_k(\underline{i}):m^{-k\alpha(1+\epsilon)}\leq\mu(R_k)\leq m^{-k\alpha(1-\epsilon)}\}.$$
For a sequence $\underline{i}\in\Sigma$ we will define
$$V_k(\underline{i}):=\inf{\big \{}l>k:\pi(i_l)\notin\{0,m-1\}\text{ or }\pi(i_l)\neq i_{m+1}{\big \}}-k-1.$$
This function will be used to bound the distance between $x$ and
the horizontal boundary of $R_k(\underline{i})$.

\begin{lemma}\label{close}
For any $x\in X_\alpha$ and $\underline{j}\in\Sigma$ such that
$\Pi\underline{j}=x$ there exists $K\in\N$ such that for any
$k\geq K$ there exists $\underline{i}\in\Sigma$ such that:
\begin{enumerate}
\item
$d(R_k(\underline{i}),x)\leq \frac{D_1 m^{-k}}{2}$,
\item
$R_k(\underline{i})\in Y(\alpha,\epsilon,k)$,
\item
If $V_k(\underline{j})\leq\frac{\epsilon k}{2}$ then we can choose
$\underline{i}$ such that $(\pi(i_1),\ldots,\pi(i_k))=(\pi(j_1),\ldots,\pi(j_k)).$
\end{enumerate}
\end{lemma}
\begin{proof}
Fix $\epsilon>0$. If $x\in X_\alpha$ then we can find $R>0$ such
that if $r<R$ we have that $\frac{\log\mu(B(x,r))}{\log r}\in
[\alpha(1-\epsilon)/3,\alpha(1+\epsilon)/3]$. We choose $k$ such
that $D_1m^{-k}\leq \frac{R}{2}$ and let $\underline{j}\in\Sigma$
satisfy $d(R_k(\underline{j}),x)\leq \frac{D_1m^{-k}}{2}$. It then
follows that $\mu(R_k(\underline{j}))\leq
(D_1m^{-k})^{\alpha(1-\epsilon/3)}$. Furthermore there are at most
$9$ sequences $\underline{j}^k$ where $d(R_k(\underline{j}),x)\leq
m^{-k}$ which means that one of these sequences $\underline{j}^k$
must satisfy $\mu(R_k(\underline{j}))\geq
\frac{m^{-k\alpha(1+\epsilon/3)}}{9}$. Parts 1 and 2 of the
assertion now follow easily.

For part 3 we fix such a $k$ and note that for any sequence
$\underline{i}$ where $(\pi(i_1),\ldots,\pi(i_k))\neq
(\pi(j_1),\ldots,\pi(j_k))$ we have that

\[
d(x,R_k(\underline{i}))\geq m^{-k-V_k(\underline{j})}  \geq
m^{-k(1+\epsilon/2)}.
\]

This means that any $\underline{i}\in \Sigma$ which satisfies
$B(x,m^{-k(1+\epsilon/2)})\cap R_k(\underline{i})\neq\emptyset$
must also satisfy
$(\pi(i_1),\ldots,\pi(i_k))=(\pi(j_1),\ldots,\pi(j_k)).$ Hence the ball $B(x,m^{-k(1+\epsilon/2)})$ satisfying
$$\mu(B(x,m^{-k(1+\epsilon/2)}))\geq m^{-k\alpha(1+\epsilon/2)(1+\epsilon/3)}\geq 2m^{-k\alpha(1+\epsilon)}$$
must be contained in the union of two approximate squares $\R_k(\underline{i})$ which both satisfy $(\pi(i_1),\ldots,\pi(i_k))=(\pi(j_1),\ldots,\pi(j_k))$. This implies that one of these approximate squares has measure not smaller than $m^{-k\alpha(1+\epsilon)}$. This together with the upper bound proved for the measure in part (2) means that one of these approximate squares is contained in $Y(\alpha,\epsilon,k)$. This completes the proof.
  \end{proof}
It should be noted that while this lemma indicates that
\[
X_\alpha \subset \bigcup_{R_k(\underline{i}) \in Y(\alpha, \epsilon,k)} B(R_k(\underline{i}), D_1 m^{-k}/2),
\]
this would not provide an efficient cover in terms of Hausdorff
dimension. To get the efficient cover we define a function
$\omega:D\rightarrow\R$ by

$$\omega(i,j)= q_i^t\gamma_i.$$
It is important that $\omega$ only depends on the vertical
coordinates. For $\underline{i}\in\Sigma$ we will denote

\[
B_l(\underline{i})=\frac {1} {l} \sum_{r=1}^l \log\omega_{i_r}
\]
and

\[
A_k(\underline{i}) = B_{l(k)}(\underline{i})-B_k(\underline{i}).
\]
$A_k$ is essentially the logarithm of the function $f_k$ defined on page 6 of \cite{king}.  The upper bound in
\cite{king} uses covering of approximate squares, $R_k(\underline{i})$, where $A_k(\underline{i})>-\epsilon$ for
some small $\epsilon$. More precisely for any $\epsilon>0$ we let
$$G(\alpha,\epsilon,k)=Y(\alpha,\epsilon,k)\cap\{R_k(\underline{i}):A_k(\underline{i})\geq-\log(1+\epsilon)\}$$
and note that in \cite{king} it is shown that for any $\underline{i}\in\Sigma$ where $\delta_{\mu}(\underline{i})=\alpha$,
there exist infinitely many $k$ such that $\R_k(\underline{i})\in G(\alpha,\epsilon,k)$. We are going to show that these covers can be modified so that the result holds for any $x\in X_{\alpha}$ even if $\underline{i}\in|\Sigma$ with $\Pi(\underline{i})=x$ does not satisfy $\delta_{\mu}(\underline{i})=\alpha$ . To be able to do this we need the following proposition:
\begin{prop}\label{keyprop}
For any $\epsilon>0$ and $x\in X_{\alpha}$ there exist infinitely
many $k\in\N$ for which there is a sequence
$\underline{i}\in\Sigma$ such that
\begin{enumerate}
\item[1.] $d(R_k(\underline{i}),x)\leq D_1 m^{-k},$

\item[2.] $A_k(\underline{i})\geq-\epsilon,$

\item[3.] $R_k(\underline{i})\in Y(\alpha,\epsilon,k).$
\end{enumerate}
\end{prop}

Before we prove Proposition \ref{keyprop}, we will show how this
proposition and the following simple lemma imply the upper bound
for $\dim X_{\alpha}$.
\begin{lemma}\label{cylinder}
If $R_k(\underline{i})\in G(\alpha,\epsilon,k)$ and $k$ is
sufficiently large then
$$\mu(R_k(\underline{i}))^t\leq (1+\epsilon)^k\sum_{(j_1,\ldots,j_k)\in\Gamma_k(\underline{i})}(p_{j_1}\cdots p_{j_k})^t(\omega_{j_1}\cdots \omega_{j_k})^{1-\sigma}.$$
\end{lemma}
\begin{proof}
If we fix $k\in\N$, $\underline{i}$ and write $l=l(k)$ then
$$\mu(R_k(\underline{i}))^t=\sum_{(j_1,\ldots,j_k)\in\Gamma_k(\underline{i})}(p_{j_1}\cdots p_{j_k})^t\omega_{j_{l+1}}\cdots \omega_{j_k}.$$
However for any $(j_1,\ldots,j_k)\in\Gamma_k(\underline{i})$ we
have that $\omega_{j_{l+1}}\cdots
\omega_{i_k}=\omega_{j_{l+1}}\cdots \omega_{i_k}$. Thus since
$A_k(\underline{i})\geq-\log(1+\epsilon)$ we have that for
$(j_1,\ldots,j_k)\in\Gamma_k(\underline{i})$
$$\frac{ \omega_{j_{l+1}}\cdots\omega_{j_k}}{(\omega_{j_1}\cdots\omega_{j_k})^{1-\sigma}}=\frac{(\omega_{i_1}\cdots\omega_{i_k})^{\sigma}}{\omega_{i_1}\cdots\omega_{i_l}}\leq e^{\sigma k B_k(\underline{i})-l B_l(\underline{i})}\leq e^{-B_k(\underline{i})}(1+\epsilon)^l.$$
To complete the proof simply note that $e^{-B_k(\underline{i})}$
is uniformly bounded by some constant $C$ and thus it is enough to
choose $k$ large enough so that $(1+\epsilon)^{k-l}\geq C$.
\end{proof}

Proposition \ref{keyprop} shows that for any $K\in \N$ and
$\epsilon>0$ we have

\[
X_\alpha \subset \bigcup_{k>K} \bigcup_{R_k(\underline{i}) \in
G(\alpha,\epsilon,k)} \widehat{R}_k(\underline{i}),
\]
where $\widehat{R}_k(\underline{i})$ stands for a rectangle with
the same center as $R_k(\underline{i})$ but $2D_1$ times greater.

We now choose $\epsilon>0$. As $m\geq 2$, we have $\log
(1+\epsilon) < 2\epsilon\log m$. For any
$\delta>\epsilon(\alpha|t|+2)$ and $K\in\N$ we have that by using
the definition of $G(\alpha,\epsilon,k)$, using Lemma
\ref{cylinder} and applying the multinomial theorem
\begin{eqnarray*}
&&\sum_{k\geq K}\sum_{R_k(\underline{i})\in G(\alpha,\epsilon,k)} |\widehat{R}_k(\underline{i})|^{\alpha t+\beta(t)+\delta}\\
&\leq& (2D_1)^{\alpha t + \beta(t)+\delta}\sum_{k\geq K}\sum_{R_k(\underline{i})\in G(\alpha,\epsilon,k)} |R_k(\underline{i})|^{\beta(t)+2\epsilon}\mu(R_k(\underline{i}))^{t}\\
&\leq&      D_2\sum_{k\geq K}\sum_{R_k(\underline{i})\in G(\alpha,\epsilon,k)} m^{-k(\beta(t)+2\epsilon)}(1+\epsilon)^k\sum_{(j_1,\ldots,j_k)\in\Gamma_k(\underline{i})}(p_{j_1}\cdots p_{j_k})^t(\omega_{j_1}\cdots \omega_{j_k})^{1-\sigma}\\
&\leq& D_2\sum_{k\geq K} m^{-2\epsilon k}(1+\epsilon)^k\sum_{(i_1,\ldots,i_k)\in D^k}m^{-k\beta(t)}(p_{i_1}\cdots p_{i_k})^t(\omega_{i_1}\cdots \omega_{i_k})^{1-\sigma}\\
&=& D_2\sum_{k\geq K} m^{-2\epsilon k}(1+\epsilon)^k \left(\sum_{(i,j)\in D} m^{-\beta(t)}p_{ij}^t\omega_{ij}^{1-\sigma}\right)^k\\
&=& D_2\sum_{k\geq K} m^{-2\epsilon k}(1+\epsilon)^k<\infty
\end{eqnarray*}
(where $D_2=2^{\alpha t + \beta(t)+\delta} D_1^{\alpha t +
2\beta(t) + \delta+ 2\epsilon}$). It follows immediately that
$$\dim X_{\alpha}\leq t\alpha+\beta(t)+\delta$$
and $\delta$ can be arbitrarily small.

We now proceed to prove Proposition \ref{keyprop}.
\subsection*{Proof of Proposition \ref{keyprop}}
We start with the case where $\underline{i}\in\Sigma$, $\Pi\underline{i}\in X_{\alpha}$ and for some $k$,
$V_k(\underline{i})=\infty$. We assume, without loss of generality, that $\pi(i_m)=0$ for all $m>k$ and that
$\pi(i_k)\neq 0$ and let $\epsilon>0$. If we fix $\eta\in\N$ and consider $k+\eta$ level approximate squares
$R_{k+\eta}(\underline{j})$ which have points within $m^{-k-\eta}/2$ of $x$ then $j$ must satisfy
$\pi(j_u)=\pi(i_u)$ for all $u\leq k+\eta$ or $\pi(j_u)=m-1$ for all $k<u\leq k+\eta$. In both of these cases if $\eta$
is sufficiently large then $A_{k+\eta}(\underline{j})\geq-\epsilon$ ($B_{k+\eta}$ is a converging sequence at such points). It follows from
parts 1 and 2 of Lemma \ref{close} that if $\eta$ is sufficiently large one of these sequences $j$ must satisfy
that $R_{k+\eta}(\underline{j})\in Y(\alpha,\epsilon,k+\eta)$.

We now turn to the case where $V_k(\underline{i})<\infty$. In this case Proposition \ref{keyprop} will follow
from the following lemma.
\begin{lemma}\label{good}
For any $\epsilon>0$ if $\underline{i}\in\Sigma$ and $V_u(\underline{i})<\infty$ for all $u$ then
we can find infinitely many $k\in\N$ such that
\begin{enumerate}
\item $A_k(\underline{i})>-\epsilon$,

\item $V_k(\underline{i})=0$.
\end{enumerate}
\end{lemma}
\begin{proof}
Let $\epsilon>0$ and $\underline{i}\in\Sigma$ such that
$V_k(\underline{i})<\infty$ for all $\underline{i}$. It is
possible to find bounds $C_1,C_2\in\R$ such that $C_1\leq
B_k(\underline{i})\leq C_2$. We now prove the assertion by
contradiction. Assuming $\underline{i}$ does not satisfy the
assertion, we can find $K\in\N$ such that for all $k\geq K$ we
have that $A_k(\underline{i})\leq-\epsilon$ or
$V_k(\underline{i})>0$.

Firstly let us consider how large $V_k(\underline{i})$ can be. By
definition, $B_{\eta}(\underline{i})$ is the sequence of Birkhoff
averages for the locally constant function
$\underline{\tau}\rightarrow \log\omega_{\tau_1}$. Since we have
that $\omega_{i_\eta}$ is constant and equal to $\omega_{i_k}$ for
$k<\eta\leq k+V_k(\underline{i})$, we have

\begin{equation*}
A_{k+V_k(\underline{i})}(\underline{i}) = B_{[\sigma
(k+V_k(\underline{i}))]}(\underline{i})-B_{k+V_k(\underline{i})}(\underline{i}) = \frac
{(1-\sigma)k} {[\sigma(k+V_k(\underline{i}))]} (B_k(\underline{i})
- \log \omega_{i_k}).
\end{equation*}
As the last multiplier cannot be smaller than $C_1-C_2$ and the
left hand side is at most $-\epsilon$ (because
$V_{k+V_k(\underline{i})}(\underline{i})=0$), we can estimate

\begin{equation} \label{bigseq}
\frac 1 k V_k(\underline{i}) \leq V = \frac {1-\sigma} {\sigma
\epsilon} (C_2-C_1)
\end{equation}
for all except possibly finitely many $k$.

We continue with the proof. We choose $u\in\N$ such that
$u\epsilon>C_2-C_1$. Denote

\[
W= 1-\frac {\log (V+1)} {\log \sigma}.
\]
We can choose $n_1>K\sigma^{-W(u+1)}$ such that
$V_{n_1}(\underline{i})=0$ (and hence
$A_{n_1}(\underline{i})\leq-\epsilon$). We can now define a
sequence $n_j$ inductively. We assume that
$A_{n_j}(\underline{i})\leq-\epsilon$ and follow the following
procedure.
\begin{enumerate}
\item If $n_j<K\sigma^{-W}$ then the procedure terminates and we
let $J=j$.

\item If $A_{[\sigma n_j]}<-\epsilon$ then let $n_{j+1}=[\sigma
n_j]$ and note that $B_{n_{j+1}}\leq B_{n_j}-\epsilon$.

\item If $A_{[\sigma n_j]}\geq-\epsilon$ then since $n_j\geq K
\sigma^{-W}$ we know that $V_{[\sigma
n_j]}(\underline{i})>0$. We let
$$a=\min\{\eta\geq [\sigma n_j]:V_{\eta}(\underline{i})=0\}\text{ and }
b=\max\{\eta\leq [\sigma n_j]:V_{\eta}(\underline{i})=0\}.$$
Now choose
$$n_{j+1}=\left\{\begin{array}{lll}a&\text{ if }&B_a(\underline{i})\leq B_b(\underline{i})\\
b&\text{ if }&B_b(\underline{i}) <
B_a(\underline{i})\end{array}\right..$$ 
Since we have that $\omega_{i_\eta}$ is constant for $a<\eta\leq b$ it follows that $B_{\eta}(\underline{i})$ is monotonic for $a\leq\eta\leq b$.
Hence at either $a$ or $b$ the value of $B$ is going to be not
greater than $B_{[\sigma n_j]}(\underline{i})$. Thus it follows
that $B_{n_{j+1}}(\underline{i})\leq
B_{n_j}(\underline{i})-\epsilon$. Moreover \eqref{bigseq} implies
that $n_{j+1}\geq\sigma^W n_j$.
\end{enumerate}
We now have a sequence $\{n_j\}_{j=1}^{J}$ such that for each
$n_j$ where $1< j\leq J$ we have that $B_{n_j}(\underline{i})\leq
B_{n_{j-1}}(\underline{i})-\epsilon$ and $\sigma^W n_j\leq
n_{j+1}$. As $n_J < K \sigma^{-W}$ and $n_1 > K \sigma^{-W(u+1)}$,
it follows that $J-1\geq u$ and we can bound
\begin{eqnarray*}
B_{n_J}(\underline{i})&\leq& B_{n_1}(\underline{i})-(J-1)\epsilon\\
&\leq&C_2-u\epsilon< C_1.
\end{eqnarray*}
This contradicts the definition of $C_1$.
\end{proof}
We now fix $x\in X_{\alpha}$ and let $\underline{i}\in\Sigma$ satisfy $\Pi\underline{i}=x$. The proof of
Proposition \ref{keyprop} is completed by combining Lemma \ref{good} with part 3 of Lemma \ref{close} (recall that $A$ only depends on the vertical coordinates).

\end{document}